\documentclass[11pt,a4paper,reqno]{amsart}
\usepackage[T1]{fontenc}
\usepackage{lmodern}
\usepackage[margin=27mm]{geometry}
\usepackage{amsmath,amssymb,amsthm,mathtools,booktabs,microtype}
\usepackage{xcolor}
\usepackage[colorlinks=true,linkcolor=blue!50!black,citecolor=blue!50!black,urlcolor=blue!50!black]{hyperref}
\hypersetup{pdftitle={Fourth-order fusion asymptotics for Sine beta correlation functions},pdfauthor={Weiyang Fang}}
\numberwithin{equation}{section}
\newtheorem{theorem}{Theorem}[section]
\newtheorem{proposition}[theorem]{Proposition}
\newtheorem{lemma}[theorem]{Lemma}
\newtheorem{corollary}[theorem]{Corollary}
\theoremstyle{definition}

\theoremstyle{remark}
\newtheorem{remark}[theorem]{Remark}
\newcommand{\E}{\mathbb E}

\newcommand{\R}{\mathbb R}
\newcommand{\C}{\mathbb C}
\newcommand{\Sine}{\mathsf{Sine}}
\newcommand{\HP}{\mathsf{HP}}

\newcommand{\dd}{\,\mathrm d}
\DeclareMathOperator{\PV}{PV}
\DeclareMathOperator{\sinc}{sinc}
\allowdisplaybreaks[2]
\title[Fourth-order fusion for Sine beta]{Fourth-order fusion asymptotics for\texorpdfstring{\ $\Sine_\beta$}{ Sine beta} correlation functions}
\author{Weiyang Fang}
\address{Independent Researcher, Tokyo, Japan}
\email{weiyang.fang@hotmail.com}
\date{September 10, 2026}
\subjclass[2020]{60B20, 60G55, 82B21}
\keywords{Sine beta process, fusion asymptotics, circular Jacobi ensemble, Hua--Pickrell process, Ward identities, inverse moments}
\begin{document}
\begin{abstract}
We compute the fourth-order correction to the full-collision asymptotics of the correlation functions of the $\Sine_\beta$ process. For $m\ge2$ and $m\beta>3$, the normalized correlation has an expansion through order $\varepsilon^4$, with an explicit rational coefficient depending on the centered profile only through its fourth power sum and the square of its second power sum. The remainder is $o(\varepsilon^4)$, locally uniformly in the collision profile. We evaluate the required fourth inverse moments of the Hua--Pickrell environment by finite-dimensional Ward identities and prove their convergence using characteristic-polynomial derivative bounds. A fourth-order expectation--Taylor lemma handles the full range $m\beta>3$ without requiring fourth moments of every analytic derivative. For general unitary ensembles with a $C^4$ confining potential and a regular bulk point, we prove convergence of the finite-particle fusion coefficients through fourth order and a joint second-order limit, using complex kernel universality and divided differences. This unitary result imposes no fused-environment hypotheses. For arbitrary beta, we retain a conditional quadratic transfer criterion.
\end{abstract}
\maketitle

\section{Introduction and main results}
The short-distance behavior of a point process records both the repulsion between the points being observed and the response of the surrounding particles. For the $\Sine_\beta$ process, the leading full-collision law is a Vandermonde power. Once that factor is removed, the next coefficients describe the interaction between the collision profile and the fused environment.

We use the normalization in which $\Sine_\beta$ has intensity $1/(2\pi)$. The stochastic-zeta representation of Assiotis and Najnudel \cite{AN} expresses its $m$-point correlation as
\begin{equation}\label{eq:correlation-representation}
 \rho_\beta^{(m)}(x_1,\ldots,x_m)
 =C_\beta^{(m)}\prod_{i<j}|x_i-x_j|^\beta\,
 \E\prod_{j=2}^m|\xi^{\beta,m\beta/2}(x_j-x_1)|^\beta.
\end{equation}
Here $\xi^{\beta,\delta}$ is the normalized Hua--Pickrell stochastic zeta function, with $\xi^{\beta,\delta}(0)=1$. The leading constant can be written, with $h=\beta/2$, as
\begin{equation}\label{eq:leading-constant}
 C_\beta^{(m)}=
 \frac{h^{hm(m-1)}\Gamma(1+h)^m}{(2\pi)^m}
 \frac{\prod_{j=0}^{m-1}\Gamma(1+jh)}{\prod_{j=m}^{2m-1}\Gamma(1+jh)}.
\end{equation}
This normalization also follows from the circular Selberg integral; see \cite{Forrester,FangCritical}.

For a pairwise distinct real profile $a=(a_1,\ldots,a_m)$, set
\begin{equation}\label{eq:profile}
 \bar a=\frac1m\sum_i a_i,\qquad b_i=a_i-\bar a,\qquad
 p_r(a)=\sum_i b_i^r,\qquad V(a)=\sum_{i<j}(a_i-a_j)^2=mp_2(a),
\end{equation}
and define
\begin{equation}\label{eq:normalized-correlation}
 R_{\beta,m}(\varepsilon;a)=
 \frac{\rho_\beta^{(m)}(\varepsilon a_1,\ldots,\varepsilon a_m)}
 {C_\beta^{(m)}|\varepsilon|^{\beta\binom m2}\prod_{i<j}|a_i-a_j|^\beta}.
\end{equation}
The second-order expansion established in \cite{FangSecond} is
\begin{equation}\label{eq:second-order}
 R_{\beta,m}(\varepsilon;a)
 =1-\frac{\beta^2V(a)}{8(m\beta-1)(2m+1)}\varepsilon^2+o(\varepsilon^2),
 \qquad m\beta>1.
\end{equation}
The companion manuscript \cite{FangCritical} treats the critical and subcritical ranges $m\beta\le1$. The present paper addresses the next regular coefficient. The new moment condition is $m\beta>3$.

Put
\begin{align}
 \mathcal D_{\beta,m}
 &=(m\beta-1)(m\beta-3)(2m+1)(2m+3)\bigl((2m+1)\beta-2\bigr),\label{eq:denominator}\\
 \mathcal T_{\beta,m}&=m(2m+3)\beta-6(m+1).\label{eq:numerator}
\end{align}
\begin{theorem}[Fourth-order fusion]\label{thm:main}
Let $m\ge2$, $\beta>0$, and $m\beta>3$. Then
\begin{equation}\label{eq:main-expansion}
 R_{\beta,m}(\varepsilon;a)
 =1-\frac{\beta^2V(a)}{8(m\beta-1)(2m+1)}\varepsilon^2
 +c_{4,\beta,m}(a)\varepsilon^4+o(\varepsilon^4),
\end{equation}
where
\begin{equation}\label{eq:main-coefficient}
 c_{4,\beta,m}(a)=
 \frac{\beta^4m}{128\mathcal D_{\beta,m}}
 \left\{\mathcal T_{\beta,m}p_2(a)^2-2m p_4(a)\right\}.
\end{equation}
The remainder is uniform for $a$ in compact sets of pairwise distinct profiles. Through the expectation in \eqref{eq:correlation-representation}, the normalized expression extends to all profiles, and the same remainder statement holds on arbitrary compact subsets of $\R^m$ for that extension.
\end{theorem}

\begin{corollary}[Two-point formula]\label{cor:pair}
For $\beta>3/2$,
\begin{align}\label{eq:pair}
 \rho_\beta^{(2)}(0,x)=C_\beta^{(2)}|x|^\beta\bigg[
 1-\frac{\beta^2}{40(2\beta-1)}x^2
 +\frac{\beta^4(7\beta-10)}{4480(2\beta-1)(2\beta-3)(5\beta-2)}x^4
 +o(x^4)\bigg].
\end{align}
\end{corollary}
\begin{proof}
For $a=(0,1)$, one has $p_2=1/2$ and $p_4=1/8$. Substitute these values in Theorem~\ref{thm:main}.
\end{proof}

At fourth order, collision geometry generally retains two invariants. For $m=2$ or $3$, the centered identity $p_4=p_2^2/2$ makes them dependent. For larger $m$, their relative size can vary. Thus the fourth coefficient carries shape information beyond the quadratic discriminant in \eqref{eq:second-order}.

The proof has two parts. A finite system of Ward identities evaluates the inverse moments $\E S_4$ and $\E S_2^2$ of the fused environment. A separate expectation--Taylor argument justifies the expansion, retaining the additional moment margin of the anchored representation \eqref{eq:correlation-representation}.

Section~\ref{sec:conditional-universality} also addresses finite-particle ensembles with general external potentials. Theorem~\ref{thm:unitary-universality} proves unconditional convergence of their quadratic and quartic fusion coefficients when $\beta=2$, the potential is $C^4$ and quadratically confining, and the equilibrium density is continuous and positive near the bulk point. The proof applies the complex kernel limit of Levin and Lubinsky~\cite{LL} to divided differences, which remain regular at a full collision. A uniform Taylor remainder also gives a simultaneous second-order limit when the exact finite-particle collision amplitude is used. For arbitrary $\beta$, Theorem~\ref{thm:univ-conditional-transfer} remains a conditional criterion; its fused-density hypotheses are not proved here for general potentials.

\section{Stochastic-zeta inputs and moment bounds}\label{sec:inputs}
Let $\alpha=m\beta$ and $\delta=\alpha/2$. The circular Jacobi ensemble with $N$ environmental particles has density proportional to
\begin{equation}\label{eq:cj-density}
 \prod_{j<k}|e^{i\theta_j}-e^{i\theta_k}|^\beta
 \prod_{j=1}^N|1-e^{i\theta_j}|^\alpha\prod_{j=1}^N\frac{\mathrm d\theta_j}{2\pi}.
\end{equation}
Under this law define
\begin{equation}\label{eq:fN}
 q_N(z)=\prod_{j=1}^N\frac{z-e^{i\theta_j}}{1-e^{i\theta_j}},\qquad
 f_N(z)=e^{-iz/2}q_N(e^{iz/N}).
\end{equation}
The function $f_N$ is entire, real on the real axis, and $f_N(0)=1$. We write $\xi=\xi^{\beta,\alpha/2}$.

\begin{proposition}[Probabilistic inputs]\label{prop:inputs}
The circular Jacobi functions can be coupled so that $f_N\to\xi$ almost surely, locally uniformly on $\C$. For every integer $k\ge0$, $R>0$, and $0<q<\alpha$,
\begin{equation}\label{eq:compact-moments}
 \sup_N\E\sup_{|z|\le R}|f_N^{(k)}(z)|^q<\infty,
 \qquad \E\sup_{|z|\le R}|\xi^{(k)}(z)|^q<\infty.
\end{equation}
For $0<p<\alpha+1$, the slope satisfies
\begin{equation}\label{eq:slope-moments}
 \sup_N\E|f_N'(0)|^p<\infty,\qquad
 \E|f_N'(0)|^p\longrightarrow\E|\xi'(0)|^p<\infty.
\end{equation}
The zero set of $\xi$ is $\HP_{\beta,\alpha/2}$, and its principal-value product is
\begin{equation}\label{eq:pv-product}
 \xi(z)=\lim_{L\to\infty}\prod_{\substack{x\in\HP_{\beta,\alpha/2}\\|x|<L}}(1-z/x).
\end{equation}
\end{proposition}
\begin{proof}
The coupling and product representation are from \cite[Propositions 2.7--2.8]{AN}, based on the circular Jacobi limit in \cite{LV}. The complex-plane estimate in \cite[proof of Proposition 2.10]{AN} gives
\[
 \E|f_N(x+iy)|^q\le C_q e^{q|y|/2},\qquad 0<q<\alpha,
\]
uniformly in $N$. Applying the Poisson inequality to the subharmonic function $|f_N|^q$ on a disk of radius $R+1$ bounds its supremum on the disk of radius $R$ by a constant times its boundary integral. Taking expectations gives the compact bound for $k=0$; Cauchy's estimate on a slightly larger disk gives every fixed $k$. Fatou's lemma gives the corresponding limiting bounds. This consequence is also recorded in \cite[Section 4]{FangSecond}. Finally,
\[
 f_N'(0)=-\frac1{2N}\sum_j\cot(\theta_j/2).
\]
The real absolute-moment convergence for the normalized Hua--Pickrell trace in \cite[Propositions 3.8 and 3.11]{AGS} gives \eqref{eq:slope-moments}. The coupling identifies the limiting trace with $-\xi'(0)$.
\end{proof}

Define
\begin{equation}\label{eq:Sk}
 S_1=\PV\sum_{x\in\HP_{\beta,\alpha/2}}x^{-1},\qquad
 S_k=\sum_{x\in\HP_{\beta,\alpha/2}}x^{-k},\quad k\ge2.
\end{equation}
The sums for $k\ge2$ are absolutely convergent almost surely. In a random neighborhood of zero with no zero of $\xi$,
\begin{equation}\label{eq:log-series}
 \log\xi(z)=-S_1z-\frac{S_2}{2}z^2-\frac{S_3}{3}z^3-\frac{S_4}{4}z^4+O(z^5).
\end{equation}
This identity will be used to identify Taylor coefficients. The interchange of the expansion and expectation is proved separately below.


\section{Finite Ward identities and fourth inverse moments}
\label{sec:ward}

Throughout this section let $h=\beta/2>0$ and $\alpha>3$.
The specialization relevant to fusion is $\alpha=m\beta$.
For $N\geq1$, the Cayley coordinates
$u_j=\cot(\theta_j/2)$ of the circular Jacobi ensemble
$\mathsf{CJ}_{N,\beta,\alpha/2}$ have density
\begin{equation}
 \frac{1}{Z_N^{\mathrm{HP}}}
 \prod_{i<j}|u_i-u_j|^{2h}
 \prod_{i=1}^N(1+u_i^2)^{-s_N},
 \qquad s_N=1+\frac{\alpha}{2}+h(N-1).
 \label{eq:ward-cauchy-density}
\end{equation}
Write $p_j=\sum_{i=1}^N u_i^j$, with $p_0=N$, and
$q_{j,N}=p_j/(2N)^j$ for $j\geq1$.

\begin{lemma}[Polynomial Ward identity]
\label{lem:polynomial-ward}
Let $1\leq r\leq4$ and let $F$ be a symmetric polynomial of total
degree at most $4-r$.  Under \eqref{eq:ward-cauchy-density},
\begin{equation}
 0=\mathbb E\left[
 F\mathcal L_{r,N}
 +\sum_{i=1}^N u_i^{r-1}(1+u_i^2)\partial_i F
 \right],
 \label{eq:polynomial-ward-general}
\end{equation}
where
\begin{align}
 \mathcal L_{1,N}&=-\alpha p_1,\notag\\
 \mathcal L_{2,N}&=(1-h-\alpha)p_2+h p_1^2+K_N,\notag\\
 \mathcal L_{3,N}&=(2-2h-\alpha)p_3+2h p_1p_2+L_Np_1,\notag\\
 \mathcal L_{4,N}&=(3-3h-\alpha)p_4+h(2p_1p_3+p_2^2)
                         +M_Np_2+h p_1^2,
 \label{eq:ward-operators}
\end{align}
and
\begin{equation}
 K_N=hN^2+(1-h)N,\qquad
 L_N=2hN+2-2h,\qquad
 M_N=2hN+3-3h.
 \label{eq:ward-finite-constants}
\end{equation}
\end{lemma}

\begin{proof}
All polynomial moments of total degree at most four are finite.
Indeed, $|x-y|^2\leq2(1+x^2)(1+y^2)$ bounds the unnormalized density
in \eqref{eq:ward-cauchy-density} by
$C_{N,h}\prod_i(1+u_i^2)^{-1-\alpha/2}$.
Each factor has every absolute moment of order strictly below
$\alpha+1$, and $4<\alpha+1$.

Here is a justification of integration by parts that also covers
$0<\beta<1$.  Set $v(x)=x^{r-1}(1+x^2)$.
Choose $\chi\in C_c^1(\mathbb R)$ with $0\leq\chi\leq1$, equal to
one on $[-1,1]$ and zero outside $[-2,2]$, and put
$\chi_R(u)=\prod_i\chi(u_i/R)$.
First replace each Vandermonde factor $|u_i-u_j|^{2h}$ by
$((u_i-u_j)^2+\eta^2)^h$, $\eta>0$.
The integral of
$\sum_i\partial_i\{v(u_i)F\chi_R\,w_{N,\eta}(u)\}$ is zero,
where $w_{N,\eta}$ denotes this smooth unnormalized density.
After summing a pair of logarithmic derivatives, its contribution is
\[
 2h F\chi_R\sum_{i<j}
 \frac{v(u_i)-v(u_j)}{u_i-u_j}
 \frac{(u_i-u_j)^2}{(u_i-u_j)^2+\eta^2}\,w_{N,\eta}(u).
\]
The difference quotient is a polynomial, and the second fraction is
between zero and one.  On the compact support of $\chi_R$, dominated
convergence therefore permits $\eta\downarrow0$ in the summed
identity.  This avoids a separate integrability assertion for the
unsymmetrized singular logarithmic derivatives.

The terms in which a derivative falls on $\chi_R$ have absolute
value bounded by
\[
 C_{F,r,N}(1+\textstyle\sum_i|u_i|)^4
 \mathbf 1_{\{\max_i|u_i|\geq R\}}\,w_N(u).
\]
To see the bound, use $|v(u_i)|/R\leq C(1+|u_i|)^r$ on
$R\leq|u_i|\leq2R$, and $r+\deg F\leq4$.
The preceding moment bound makes this majorant integrable, so the
cutoff-derivative terms vanish as $R\to\infty$.
All other terms, after pairwise symmetrization, are polynomials of
total degree at most four and also pass to the limit by dominated
convergence.  Division by $Z_N^{\mathrm{HP}}$ gives
\[
0=\mathbb E\left[
 F\left\{\sum_i v'(u_i)-2s_Np_r
       +2h\sum_{i<j}\frac{v(u_i)-v(u_j)}{u_i-u_j}\right\}
 +\sum_i v(u_i)\partial_iF\right].
\]
For a monomial of positive degree $k$,
\[
 \sum_{i<j}\frac{u_i^k-u_j^k}{u_i-u_j}
 =\frac12\sum_{\ell=0}^{k-1}
       (p_\ell p_{k-1-\ell}-p_{k-1}).
\]
Apply this identity to the two monomials in $v$; the constant
monomial when $r=1$ contributes zero.  Substituting the value of
$s_N$ yields \eqref{eq:ward-operators}.
\end{proof}

For clarity, we give the exact finite-dimensional system before
taking a limit.  Define
\begin{align*}
 a_N&=\mathbb E q_{2,N},&
 b_N&=\mathbb E q_{1,N}^2,\\
 A_N&=\mathbb E q_{4,N},&
 B_N&=\mathbb E(q_{3,N}q_{1,N}),&
 C_N&=\mathbb E q_{2,N}^2,\\
 D_N&=\mathbb E(q_{2,N}q_{1,N}^2),&
 E_N&=\mathbb E q_{1,N}^4.
\end{align*}
The two choices $(r,F)=(1,p_1),(2,1)$ give
\[
 \alpha b_N=a_N+\frac1{4N},\qquad
 (1-h-\alpha)a_N+h b_N+k_N=0,
 \quad k_N=\frac h4+\frac{1-h}{4N}.
\]
In particular,
\begin{equation}
 a_N=\frac{h\alpha N+\alpha(1-h)+h}
          {4N(\alpha-1)(\alpha+h)},\qquad
 b_N=\frac{hN+\alpha}{4N(\alpha-1)(\alpha+h)}.
 \label{eq:ward-exact-second-moments}
\end{equation}
Put
\[
 \ell_N=\frac{h}{2N}+\frac{1-h}{2N^2},\qquad
 \mu_N=\frac{h}{2N}+\frac{3-3h}{4N^2}.
\]
The choices
\[
 (r,F)=(4,1),(3,p_1),(2,p_2),(2,p_1^2),
                    (1,p_3),(1,p_2p_1),(1,p_1^3)
\]
give, in this order, the following seven exact identities:
\begin{align}
 0={}&(3-3h-\alpha)A_N+h(2B_N+C_N)
                   +\mu_Na_N+\frac{h}{4N^2}b_N,\notag\\
 0={}&A_N+(2-2h-\alpha)B_N+2hD_N
                   +\frac{a_N}{4N^2}+\ell_Nb_N,\notag\\
 0={}&2A_N+(1-h-\alpha)C_N+hD_N
                   +\left(k_N+\frac1{2N^2}\right)a_N,\notag\\
 0={}&2B_N+(1-h-\alpha)D_N+hE_N
                   +\left(k_N+\frac1{2N^2}\right)b_N,\notag\\
 0={}&-\alpha B_N+3A_N+\frac{3a_N}{4N^2},\notag\\
 0={}&-\alpha D_N+2B_N+C_N
                   +\frac{b_N}{2N^2}+\frac{a_N}{4N},\notag\\
 0={}&-\alpha E_N+3D_N+\frac{3b_N}{4N}.
 \label{eq:ward-exact-fourth-system}
\end{align}
Thus all terms omitted from the limiting system below are explicit;
by \eqref{eq:ward-exact-second-moments}, they are $O(N^{-1})$.

\subsection{Uniform integrability and the Hua--Pickrell limit}

Let
\[
 f_N(z)=e^{-iz/2}\prod_{j=1}^N
             \frac{e^{iz/N}-e^{i\theta_j}}{1-e^{i\theta_j}}
       =\prod_{j=1}^N
             \frac{\sin(\theta_j/2-z/(2N))}{\sin(\theta_j/2)}.
\]
Proposition~\ref{prop:inputs} supplies the locally uniform almost-sure coupling and the moment bounds
\begin{align}
 \sup_N\E|f_N'(0)|^p&<\infty &&(0<p<\alpha+1),\label{eq:ward-trace-input}\\
 \sup_N\E|f_N''(0)|^q&<\infty &&(0<q<\alpha).\label{eq:ward-derivative-input}
\end{align}
Only these exponent ranges are used in the passage to the limit.

For the zeros of $\xi=\xi^{\beta,\alpha/2}$, put
\[
 S_1=\operatorname{PV}\sum_x x^{-1},\qquad
 S_k=\sum_x x^{-k}\quad(k\geq2).
\]
The sums for $k\geq2$ converge absolutely almost surely: the
canonical product has $\sum_x x^{-2}<\infty$, and its zeros cannot
accumulate at zero since $\xi(0)=1$.
In a neighborhood of the origin,
\[
 (\log\xi)'(0)=-S_1,\qquad
 (\log\xi)^{(k)}(0)=-(k-1)!S_k\quad(k\geq2).
\]
The trigonometric product for $f_N$ gives the exact relations
\begin{align}
 (\log f_N)'(0)&=-q_{1,N},\notag\\
 (\log f_N)''(0)&=-q_{2,N}-\frac1{4N},\notag\\
 (\log f_N)'''(0)&=-2q_{3,N}-\frac{q_{1,N}}{2N^2},\notag\\
 (\log f_N)''''(0)&=-6q_{4,N}
                         -\frac{2q_{2,N}}{N^2}-\frac1{8N^3}.
 \label{eq:ward-log-derivatives}
\end{align}
No global logarithm is required in these identities: since
$f_N(0)=\xi(0)=1$, each derivative is a polynomial in the ordinary
derivatives at zero.  Local uniform convergence of the entire
functions and \eqref{eq:ward-log-derivatives} imply
\begin{equation}
 q_{k,N}\longrightarrow S_k\quad\hbox{almost surely},
                   \qquad 1\leq k\leq4.
 \label{eq:ward-almost-sure-limit}
\end{equation}

\begin{lemma}[Convergence of the fourth-order mixed moments]
\label{lem:ward-fourth-ui}
Under $\alpha>3$, the seven expectations defining
$a_N,b_N,A_N,B_N,C_N,D_N,E_N$ converge to the expectations of the
corresponding products of $S_1,S_2,S_3,S_4$.
\end{lemma}

\begin{proof}
The second identity of \eqref{eq:ward-log-derivatives} reads
\begin{equation}
 q_{2,N}=f_N'(0)^2-f_N''(0)-\frac1{4N}.
 \label{eq:ward-q2-ui-identity}
\end{equation}
Choose $\eta>0$ so small that
$4+2\eta<\alpha+1$ and $2+\eta<\alpha$.
Because $q_{2,N}\geq0$, equations
\eqref{eq:ward-trace-input}--\eqref{eq:ward-q2-ui-identity} imply
\begin{equation}
 \sup_N\mathbb E q_{2,N}^{2+\eta}<\infty.
 \label{eq:ward-q2-higher-bound}
\end{equation}
For any real numbers $v_i$,
$\sum_i v_i^4\leq(\sum_i v_i^2)^2$ and
$|\sum_i v_i^3|\leq(\sum_i v_i^2)^{3/2}$.
Applied to $v_i=u_i/(2N)$, these give
\begin{equation}
 0\leq q_{4,N}\leq q_{2,N}^2,\qquad
 |q_{3,N}|\leq q_{2,N}^{3/2}.
 \label{eq:ward-power-sum-bounds}
\end{equation}
Now choose $\varepsilon>0$ with
$2(1+\varepsilon)\leq2+\eta$ and
$4(1+\varepsilon)<\alpha+1$.
The bounds already proved control the $(1+\varepsilon)$-moments of
$q_{4,N}$ and $q_{2,N}^2$, while
\eqref{eq:ward-trace-input} controls those of $q_{1,N}^4$.
For the remaining products, H\"older's inequality gives
\begin{align*}
 \mathbb E|q_{3,N}q_{1,N}|^{1+\varepsilon}
 &\leq
 \left(\mathbb E q_{2,N}^{2(1+\varepsilon)}\right)^{3/4}
 \left(\mathbb E|q_{1,N}|^{4(1+\varepsilon)}\right)^{1/4},\\
 \mathbb E|q_{2,N}q_{1,N}^2|^{1+\varepsilon}
 &\leq
 \left(\mathbb E q_{2,N}^{2(1+\varepsilon)}\right)^{1/2}
 \left(\mathbb E|q_{1,N}|^{4(1+\varepsilon)}\right)^{1/2}.
\end{align*}
Both right-hand sides are bounded uniformly in $N$.
Consequently all five fourth-order products are uniformly
integrable; the second-order products are covered by the same
bounds.  Combine this with \eqref{eq:ward-almost-sure-limit} to
obtain the asserted expectation convergence and finiteness of all
limiting moments.
\end{proof}

\begin{proposition}[Fourth inverse moments]
\label{prop:fourth-inverse-moments}
Let
\[
 \mathcal D=(\alpha-1)(\alpha+h)(\alpha-1+h)(\alpha-3)(\alpha+3h),
 \qquad Q=\alpha^2-3\alpha+3h\alpha-6h.
\]
For $\mathsf{HP}_{\beta,\alpha/2}$ with $\alpha>3$,
\begin{equation}
 \mathbb E S_4=\frac{h^3\alpha^2}{16\mathcal D},\qquad
 \mathbb E S_2^2=\frac{h^2\alpha Q}{16\mathcal D}.
 \label{eq:fourth-inverse-moments-alpha}
\end{equation}
For $\alpha=m\beta$, define
\[
 \mathcal D_m=(m\beta-1)(m\beta-3)(2m+1)(2m+3)
                         ((2m+1)\beta-2),
 \qquad T_m=m(2m+3)\beta-6(m+1).
\]
Then \eqref{eq:fourth-inverse-moments-alpha} becomes
\begin{equation}
 \mathbb E S_4=\frac{\beta^3m^2}{16\mathcal D_m},\qquad
 \mathbb E S_2^2=\frac{\beta^2mT_m}{16\mathcal D_m}.
 \label{eq:fourth-inverse-moments-m}
\end{equation}
\end{proposition}

\begin{proof}
Denote the limiting fourth-order expectations by $A,B,C,D,E$,
in the order used above.  By
\eqref{eq:ward-exact-second-moments} and
Lemma~\ref{lem:ward-fourth-ui},
\[
 a:=\mathbb ES_2=\frac{h\alpha}{4(\alpha-1)(\alpha+h)},\qquad
 b:=\mathbb ES_1^2=\frac{h}{4(\alpha-1)(\alpha+h)}.
\]
Pass to the limit in \eqref{eq:ward-exact-fourth-system} to obtain
\begin{align}
 (3-3h-\alpha)A+h(2B+C)&=0,\notag\\
 A+(2-2h-\alpha)B+2hD&=0,\notag\\
 2A+(1-h-\alpha)C+hD+\frac h4a&=0,\notag\\
 2B+(1-h-\alpha)D+hE+\frac h4b&=0,\notag\\
 \alpha B=3A,\qquad \alpha D=2B+C,\qquad \alpha E&=3D.
 \label{eq:ward-limit-fourth-system}
\end{align}
The first equation and $B=3A/\alpha$ give
\[
 C=\frac{Q}{\alpha h}A.
\]
Then the equation for $D$ gives
\[
 D=\frac{\alpha-3+3h}{\alpha h}A.
\]
Substitute these two expressions into the third equation of
\eqref{eq:ward-limit-fourth-system}.
Its coefficient of $A$ is
\[
 2+(1-h-\alpha)\frac{Q}{\alpha h}
       +\frac{\alpha-3+3h}{\alpha}
 =-\frac{(\alpha-1+h)(\alpha-3)(\alpha+3h)}{\alpha h}.
\]
Every factor in the denominator used to solve for $A$ is positive
when $\alpha>3$ and $h>0$.  It follows that
\[
 A=\frac{\alpha h^2a}
          {4(\alpha-1+h)(\alpha-3)(\alpha+3h)}
   =\frac{h^3\alpha^2}{16\mathcal D}.
\]
The formula for $C$ follows.  Finally substitute
$\alpha=m\beta$ and $h=\beta/2$ to obtain
\eqref{eq:fourth-inverse-moments-m}.
\end{proof}

\begin{remark}
The same calculation records the mixed moments needed if the
fourth Taylor coefficient is computed in anchored coordinates:
\[
 \mathbb E(S_3S_1)=\frac3\alpha\mathbb ES_4,\qquad
 \mathbb E(S_2S_1^2)=
       \frac{\alpha-3+3h}{\alpha h}\mathbb ES_4,\qquad
 \mathbb ES_1^4=
       \frac{3(\alpha-3+3h)}{\alpha^2h}\mathbb ES_4.
\]
These are expectation identities under the stated integrability
condition; no endpoint moment assertion for a product of
stochastic-zeta factors is used in their proof.
\end{remark}


\section{A fourth-order expectation--Taylor lemma}
\label{sec:taylor}

The range $\alpha>3$ is obtained by distinguishing the first derivative from the
higher derivatives.  In the application the first derivative has moments of
orders less than $\alpha+1$, while compact suprema of higher derivatives have
moments of orders less than $\alpha$.  A fourth-order expansion therefore
requires a truncation argument when $3<\alpha\leq4$.

\begin{proposition}[Fourth-order expectation--Taylor lemma]
\label{prop:fourth-taylor}
Let $K\subset\mathbb R^d$ be compact and let $F:\mathbb R^d\to\mathbb R$
be measurable, $C^4$ on a neighbourhood of $0$, and satisfy
\[
 |F(u)|\leq C(1+\|u\|^D),\qquad u\in\mathbb R^d,
\]
for some $D>0$.  Suppose that
\[
 p>\max\{4,D\},\qquad q>\max\{3,D\},
 \qquad X\in L^p,\qquad Y,Z,W,M\in L^q,
\]
where $M\geq0$.  Assume that, uniformly for $b\in K$ and $1\leq j\leq d$,
\[
 U_j(t,b)=tb_jX+\frac{t^2b_j^2}{2}Y
       +\frac{t^3b_j^3}{6}Z+\frac{t^4b_j^4}{24}W+R_j(t,b),
 \qquad |R_j(t,b)|\leq C_K|t|^5M.
\]
Define random vectors
\[
 v_1=(b_jX)_j,\quad v_2=(b_j^2Y/2)_j,
 \quad v_3=(b_j^3Z/6)_j,\quad v_4=(b_j^4W/24)_j.
\]
Then, uniformly for $b\in K$ as $t\to0$,
\begin{align*}
 \mathbb EF(U(t,b))={}&F(0)+t\,\mathbb E DF(0)[v_1]\\
 &+t^2\mathbb E\left\{DF(0)[v_2]
                      +\tfrac12D^2F(0)[v_1,v_1]\right\}\\
 &+t^3\mathbb E\left\{DF(0)[v_3]+D^2F(0)[v_1,v_2]
                      +\tfrac16D^3F(0)[v_1,v_1,v_1]\right\}\\
 &+t^4\mathbb E\left\{DF(0)[v_4]+D^2F(0)[v_1,v_3]
                      +\tfrac12D^2F(0)[v_2,v_2]\right.\\
 &\hspace{12mm}\left.
                      +\tfrac12D^3F(0)[v_1,v_1,v_2]
                      +\tfrac1{24}D^4F(0)[v_1,v_1,v_1,v_1]\right\}
   +o(t^4).
\end{align*}
Every expectation in this formula is absolutely integrable.
\end{proposition}

\begin{proof}
Write $s=|t|$ and $B=|Y|+|Z|+|W|+M$.  All constants below can be
chosen uniformly on $K$.  For a sufficiently large deterministic $C_K$ set
\[
 A_s=C_Ks|X|,\qquad B_s=C_Ks^2B,
 \qquad G_s=\{A_s\leq\eta_s,\ B_s\leq\eta_s\},
 \qquad \eta_s=s^\theta,
\]
where $\theta>0$ will be chosen sufficiently small.  For $s\leq1$ we have
$\|U(t,b)\|\leq A_s+B_s$.

We first record the estimates needed to remove this event.  For nonnegative
integers $a,b$ satisfying $a/p+b/q<1$, H\"older's inequality gives
\begin{align}
 s^k\mathbb E[|X|^aB^b;\,A_s>\eta_s]
 &\leq C s^{k+(1-\theta)\{p(1-b/q)-a\}},
 \label{eq:tail-slope}\\
 s^k\mathbb E[|X|^aB^b;\,B_s>\eta_s]
 &\leq C s^{k+(2-\theta)\{q(1-a/p)-b\}}.
 \label{eq:tail-higher}
\end{align}
For example, to prove \eqref{eq:tail-slope} when $b>0$, apply H\"older
with exponents $q/b$ and $q/(q-b)$, then use
$\mathbb E[|X|^r;|X|>L]\leq L^{r-p}\mathbb E|X|^p$
with $r=aq/(q-b)<p$.  The cases $a=0$ or $b=0$ follow by the same
argument, using a probability in place of the zeroth moment.

The monomials of weighted order less than four appearing in the desired
expansion are bounded by the following list.  The last two columns give
the exponents in \eqref{eq:tail-slope}--\eqref{eq:tail-higher} at
$\theta=0$:
\[
\begin{array}{c|c|c}
\text{monomial}&\text{slope-tail exponent}&\text{higher-tail exponent}\\ \hline
1&p&2q\\
s|X|&p&1+2q(1-1/p)\\
s^2X^2&p&2+2q(1-2/p)\\
s^2B&2+p(1-1/q)&2q\\
s^3|X|^3&p&3+2q(1-3/p)\\
s^3|X|B&2+p(1-1/q)&1+2q(1-1/p)\\
s^3B&3+p(1-1/q)&1+2q
\end{array}
\]
Every entry exceeds four when $p>4$ and $q>3$.  Consequently one can
choose $0<\theta<1$ so that all these exponents still exceed four and
also $p(1-\theta)>4$, $q(2-\theta)>4$.
The fourth-order monomials are bounded by multiples of
\[
 s^4\bigl(|X|^4+X^2B+B^2+|X|B+B\bigr).
\]
Their random factors are integrable: in particular
$2/p+1/q<1$.  Since $\mathbf1_{G_s^c}\to0$ almost surely, dominated
convergence removes their bad-event restrictions with error $o(s^4)$.

The actual function also has a negligible bad-event contribution.
Partition $G_s^c$ according to which of $A_s,B_s$ is larger.  On the
first piece $\|U\|\leq2A_s$ and $A_s>\eta_s$, and on the second
$\|U\|\leq2B_s$ and $B_s>\eta_s$.  The constant part of the growth
bound contributes $o(s^4)$ by the probability estimates above.  Because
$p>D$ and $q>D$, the remaining part is bounded by
\begin{align*}
 \mathbb E[A_s^D;A_s>\eta_s]
 &\leq C s^p\eta_s^{D-p}=o(s^4),\\
 \mathbb E[B_s^D;B_s>\eta_s]
 &\leq C s^{2q}\eta_s^{D-q}=o(s^4).
\end{align*}
Thus $\mathbb E[|F(U)|;G_s^c]=o(s^4)$.

On $G_s$, Taylor's theorem for $F$ gives
\[
 F(U)=\sum_{r=0}^4\frac1{r!}D^rF(0)[U^{\otimes r}]
           +\mathcal R_F(U),
 \qquad |\mathcal R_F(U)|\leq\omega(\|U\|)\|U\|^4,
\]
\[
 \qquad \omega(v)\longrightarrow0\quad(v\downarrow0).
\]
Here $\mathbb E[A_s^4]=O(s^4)$ and
$\mathbb E[B_s^4;G_s]=o(s^4)$.  For the latter assertion, if $q<4$,
\[
 \mathbb E[B_s^4;G_s]
 \leq C s^{2q}\eta_s^{4-q}\mathbb EB^q=o(s^4),
\]
and if $q\geq4$ it follows from $\mathbb E B_s^4=O(s^8)$.
The expected Taylor remainder on $G_s$ is therefore $o(s^4)$.

We finally expand the Taylor polynomial in powers of $t$ and discard
weighted orders greater than four.  Here $X$ has weight one and every
occurrence of $B$ has weight at least two.  All such discarded terms
whose actual power of $s$ exceeds four but whose minimum weight is at
most four have integrable random factors among
$B$, $|X|B$, $X^2B$ and $B^2$, so contribute $O(s^5)$.
It remains to consider
\[
 s^{a+2b}\mathbb E[|X|^a B^b;G_s],
 \qquad a+b\leq4,\quad a+2b>4.
\]
If $a/p+b/q\leq1$, this is $O(s^{a+2b})=o(s^4)$ by H\"older.
Otherwise, using H\"older on $X$ and truncating $B$ at
$C\eta_s/s^2$ gives
\[
 s^{a+2b}\mathbb E[|X|^a B^b;G_s]
 \leq C s^{a+2q(1-a/p)}
                 \eta_s^{\,b-q(1-a/p)}.
\]
The exponent of $\eta_s$ is positive.  For $a=0,1,2,3$ the exponents
of $s$ in this display are respectively
\[
 2q,\quad 1+2q(1-1/p),\quad 2+2q(1-2/p),
 \quad 3+2q(1-3/p),
\]
all exceeding four.  This proves that all discarded terms are
$o(s^4)$.  Removing $G_s$ from the retained terms by the preceding
tail estimates proves the claimed expansion.  All remainder bounds are
uniform on $K$.
\end{proof}

\begin{remark}[Uniform families]
\label{rem:uniform-taylor}
The conclusion of Proposition~\ref{prop:fourth-taylor} is uniform over an
additional index $N$ if the pathwise remainder constants and the $L^p$ and
$L^q$ bounds are uniform in $N$.  Indeed, all bad-event estimates above
are already uniform.  For the fourth-order monomials, the strict
inequalities $4/p<1$, $2/p+1/q<1$ and $2/q<1$ give a uniform
$L^{1+\epsilon}$ bound for some $\epsilon>0$, which replaces the
individual dominated-convergence step by uniform integrability.
\end{remark}

\section{Proof of the fourth-order fusion theorem}\label{sec:main-proof}
We use the moment identities from Section~\ref{sec:ward} and the expectation--Taylor lemma from Section~\ref{sec:taylor}. Put $d=m-1$ and
\[
 v_j=a_{j+1}-a_1,\quad 1\le j\le d,\qquad
 P_r=\sum_{j=1}^d v_j^r.
\]
The total growth exponent of $\prod_{j=1}^d|1+u_j|^\beta$ is $d\beta=\alpha-\beta<\alpha$. Since $\alpha>3$, exponents may be chosen so that
\begin{equation}\label{eq:exponent-choice}
 \max\{4,d\beta\}<p<\alpha+1,\qquad
 \max\{3,d\beta\}<q<\alpha.
\end{equation}
Apply the expectation--Taylor lemma with
\[
 X=\xi'(0),\quad Y=\xi''(0),\quad Z=\xi'''(0),\quad W=\xi''''(0),
 \quad M=\sup_{|z|\le R}|\xi^{(5)}(z)|,
\]
using Proposition~\ref{prop:inputs}. For profiles in a fixed compact set, the same $R$ and all Taylor constants can be chosen uniformly.

The pathwise logarithmic expansion \eqref{eq:log-series} identifies the coefficients as those of
\[
 \exp\left(-\beta\sum_{r=1}^4\frac{P_rS_r}{r}t^r\right)
 \quad\text{through degree four}.
\]
Reflection symmetry of the environment makes the first- and third-order expected coefficients zero. The second-order coefficient is
\[
 \frac{\beta^2P_1^2}{2}\E S_1^2-\frac{\beta P_2}{2}\E S_2
 =-\frac{\beta}{2}\left(P_2-\frac{P_1^2}{m}\right)\E S_2.
\]
Here $\E S_2=\alpha\E S_1^2$, and $P_2-P_1^2/m=p_2(a)$. The degree-four coefficient is
\begin{align}\label{eq:anchored-c4}
 &\frac{\beta^4P_1^4}{24}\E S_1^4
 -\frac{\beta^3P_1^2P_2}{4}\E S_1^2S_2
 +\frac{\beta^2P_2^2}{8}\E S_2^2\nonumber\\
 &\hspace{18mm}+\frac{\beta^2P_1P_3}{3}\E S_1S_3
 -\frac{\beta P_4}{4}\E S_4.
\end{align}
The Ward relations
\begin{equation}\label{eq:mixed-reduction}
 \alpha\E S_1S_3=3\E S_4,\qquad
 \alpha\E S_1^2S_2=2\E S_1S_3+\E S_2^2,\qquad
 \alpha\E S_1^4=3\E S_1^2S_2
\end{equation}
reduce \eqref{eq:anchored-c4} to
\begin{equation}\label{eq:centered-c4}
 \frac{\beta^2p_2(a)^2}{8}\E S_2^2-\frac{\beta p_4(a)}4\E S_4.
\end{equation}
Indeed,
\[
 p_4(a)=P_4-\frac4mP_1P_3+\frac6{m^2}P_1^2P_2-\frac3{m^3}P_1^4.
\]
The explicit moments of Section~\ref{sec:ward} now give \eqref{eq:main-coefficient}, while their degree-two counterparts give \eqref{eq:second-order}. Equation~\eqref{eq:correlation-representation} and the uniform remainder in the Taylor lemma complete the proof of Theorem~\ref{thm:main}.

This calculation keeps the total power in the analytic argument strictly below $\alpha$. Centering the final polynomial is an algebraic use of \eqref{eq:mixed-reduction}; it does not require an expectation estimate for $m$ unanchored factors at the endpoint total power $\alpha$.

\section{Classical checks and the fourth-moment boundary}\label{sec:checks}
For $\beta=2$, the exact pair correlation is
\[
 \rho_2^{(2)}(0,x)=\frac1{4\pi^2}\left[1-\left(\frac{\sin(x/2)}{x/2}\right)^2\right].
\]
Its normalized expansion is
\begin{equation}\label{eq:beta2-check}
 R_{2,2}(x;0,1)=1-\frac{x^2}{30}+\frac{x^4}{1680}+O(x^6).
\end{equation}
For $\beta=4$, write $\sinc x=\sin x/x$. The classical formula \cite{Forrester,QV} is
\[
 \rho_4^{(2)}(0,x)=\frac1{4\pi^2}
 \left[1-\sinc^2 x+(\sinc x)'\int_0^x\sinc t\dd t\right].
\]
The bracket equals $\frac{x^4}{135}(1-\frac2{35}x^2+\frac2{1225}x^4+O(x^6))$, and hence
\begin{equation}\label{eq:beta4-check}
 R_{4,2}(x;0,1)=1-\frac2{35}x^2+\frac2{1225}x^4+O(x^6).
\end{equation}
Both fourth-order coefficients agree with Corollary~\ref{cor:pair}.

For $\beta=2$ and arbitrary $m$, the determinant formula is
\[
 \rho_2^{(m)}(x_1,\ldots,x_m)
 =\det\!\left[\frac{\sin((x_i-x_j)/2)}{\pi(x_i-x_j)}\right]_{i,j=1}^m,
\]
with diagonal entries $1/(2\pi)$. Expansion with rational Taylor coefficients gives the following normalized coefficients. The different four-point profiles test both fourth-degree invariants.
\begin{center}
\begin{tabular}{@{}ccc@{}}
\toprule
$m$ & Profile $a$ & Coefficient of $\varepsilon^4$\\
\midrule
3 & $(0,1,2)$ & $1/280$\\
3 & $(0,1,3)$ & $7/360$\\
4 & $(0,1,2,3)$ & $19/1540$\\
4 & $(0,1,2,4)$ & $3/80$\\
4 & $(0,1,3,7)$ & $22399/55440$\\
\bottomrule
\end{tabular}
\end{center}
Every entry agrees with \eqref{eq:main-coefficient}. These calculations are consistency checks; the proof for arbitrary $\beta$ is the Ward and Taylor argument above.

The denominator in \eqref{eq:main-coefficient} contains $m\beta-3$. This boundary is consistent with the local fourth inverse-moment integral $\int_0^1x^{m\beta-4}\dd x$. The condition $m\beta>3$ is a sufficient range for the present proof; we do not prove that it is necessary for a fourth-order expansion. Theorem~\ref{thm:main} makes no assertion at the boundary itself. In particular, the apparent pole of a combined coefficient can cancel. At $m=3$, $\beta=1$, one has $\mathcal T_{\beta,m}=3$ and $p_4=p_2^2/2$, so the numerator of \eqref{eq:main-coefficient} vanishes. Such cancellation does not justify taking a limit in the asymptotic remainder or rule out other nonanalytic terms.

\section{Fusion coefficients for general external potentials}
\label{sec:conditional-universality}

Bulk universality for general beta ensembles is established in settings such as \cite{BEY}. Merging singularities in unitary ensembles have also been studied by analytic methods \cite{CF,CK}. We first identify the exact finite-particle quadratic coefficient for general $\beta$. We then prove its universality, together with that of the quartic coefficient, for the unitary class under ordinary bulk hypotheses. For arbitrary $\beta$, a separate transfer criterion specifies the additional fused-environment estimates still required.

\subsection{Finite ensembles and their quadratic coefficients}\label{subsec:finite-general}

Fix $m\geq2$, $\beta>0$ with $\alpha:=m\beta>1$, and a potential
$U\in C^2(\mathbb R)$ satisfying
\begin{equation}
 U(y)\geq c y^2-C\qquad (y\in\mathbb R)
 \label{eq:univ-confinement}
\end{equation}
for some $c>0$ and $C<\infty$.  For $N\geq m+1$, consider the density
\begin{equation}
 p_N(y_1,\ldots,y_N)
 =\frac{1}{Z_N}
  \exp\!\left\{-\frac{\beta N}{2}\sum_{j=1}^N U(y_j)\right\}
  |\Delta_N(y)|^\beta
 \label{eq:univ-ensemble}
\end{equation}
with respect to Lebesgue measure on $\mathbb R^N$.
Here $\Delta_k(y)=\prod_{i<j}(y_j-y_i)$, and $\rho_N^{(m)}$ denotes the
factorial $m$-point correlation density.  Let $E$ be a bulk point at which
the equilibrium density $\varrho_U(E)$ exists and is strictly positive.
Set
\begin{equation}
 \ell_N=\frac{1}{2\pi N\varrho_U(E)},\qquad
 \widehat\rho_N^{(m)}(x_1,\ldots,x_m)
 =\ell_N^m\rho_N^{(m)}(E+\ell_Nx_1,\ldots,E+\ell_Nx_m).
 \label{eq:univ-unfolding}
\end{equation}
This convention matches the intensity $1/(2\pi)$ used for
$\mathsf{Sine}_\beta$.

Write $M=N-m$.  The finite fused environmental law $Q_{N,m,E}$ is the
probability measure on $\mathbb R^M$ with density
\begin{equation}
 \frac{1}{Z_{N,m,E}^{\mathrm{fus}}}
 \exp\!\left\{-\frac{\beta N}{2}\sum_{j=1}^M U(y_j)\right\}
 |\Delta_M(y)|^\beta\prod_{j=1}^M|y_j-E|^{m\beta}.
 \label{eq:univ-fused-law}
\end{equation}
In particular, the coefficient of $U$ is $N$, not $N-m$.
The unfolded environmental points are $x_j=(y_j-E)/\ell_N$.
Let $r_N(x)$ be their one-point intensity under $Q_{N,m,E}$, and put
\begin{equation}
 S_{2,N}=\sum_{j=1}^M x_j^{-2}.
 \label{eq:univ-S2}
\end{equation}

For a pairwise distinct centered profile $b=(b_1,\ldots,b_m)$, define
\begin{equation}
 \sum_{i=1}^m b_i=0,\qquad B(b)=\sum_{i=1}^m b_i^2,
 \qquad \mathcal V(b)=\sum_{i<j}(b_i-b_j)^2=mB(b),
 \qquad d=\beta\binom m2.
 \label{eq:univ-profile}
\end{equation}
The exact finite-$N$ collision amplitude is
\begin{equation}
 A_N(E)=\frac{(N)_m}{Z_N}\,
  \ell_N^{m+d}\exp\!\left\{-\frac{\beta Nm}{2}U(E)\right\}
  Z_{N,m,E}^{\mathrm{fus}}>0,
 \label{eq:univ-amplitude}
\end{equation}
where $(N)_m=N!/(N-m)!$.  We normalize by this amplitude:
\begin{equation}
 R_N(\varepsilon;b)=
 \frac{\widehat\rho_N^{(m)}(\varepsilon b_1,\ldots,\varepsilon b_m)}
 {A_N(E)|\varepsilon|^d|\Delta_m(b)|^\beta}
 \quad(\varepsilon\ne0),\qquad R_N(0;b)=1.
 \label{eq:univ-normalized-R}
\end{equation}

\begin{proposition}[Exact finite-particle quadratic coefficient]
\label{prop:univ-finite-coefficient}
Under the preceding assumptions, for each fixed $N$,
\begin{equation}
 R_N(\varepsilon;b)=1-c_N(b)\varepsilon^2+o_N(\varepsilon^2),
 \label{eq:univ-finite-expansion}
\end{equation}
where
\begin{equation}
 c_N(b)=\frac{\beta B(b)}2\,
       \mathbb E_{Q_{N,m,E}}S_{2,N}
       +\frac{\beta N\ell_N^2B(b)}4\,U''(E).
 \label{eq:univ-finite-coefficient}
\end{equation}
In particular, $c_N(b)$ is the limit of
$(1-R_N(\varepsilon;b))/\varepsilon^2$ as $\varepsilon\to0$.
\end{proposition}

\begin{proof}
Dividing the defining correlation integral by
\eqref{eq:univ-amplitude} gives the exact identity
\begin{align}
 R_N(\varepsilon;b)
 &=\exp\!\left\{-\frac{\beta N}{2}
       \sum_{i=1}^m[U(E+\ell_N\varepsilon b_i)-U(E)]\right\}
 \nonumber\\
 &\hspace{1em}\times
 \mathbb E_{Q_{N,m,E}}
  \prod_{j=1}^M\prod_{i=1}^m
       \left|1-\frac{\varepsilon b_i}{x_j}\right|^\beta.
 \label{eq:univ-exact-ratio}
\end{align}
The external-potential factor is
$1-\beta N\ell_N^2B(b)U''(E)\varepsilon^2/4+o_N(\varepsilon^2)$.

We justify the quadratic expansion of the integral at the fused
configuration, without differentiating twice at separated singularities.
For $t\in\mathbb R$, set
$g_t(z)=\prod_i|z-tb_i|^\beta$.  Since $\sum_i b_i=0$, one has the
weighted $L^1$ expansion
\begin{equation}
 g_t(z)=|z|^\alpha
       -\frac{\beta B(b)}2t^2|z|^{\alpha-2}
       +o_{L^1(w(z)\,dz)}(t^2),
 \label{eq:univ-L1-taylor}
\end{equation}
for every weight $w(z)=(1+|z|)^K e^{-a z^2}$ with $a>0$ and $K\geq0$.
Indeed, on $|z|\leq C_b|t|$, the integral of the difference in
\eqref{eq:univ-L1-taylor}, after division by $t^2$, is
$O(|t|^{\alpha-1})$.  On the complement, Taylor expansion of
$\prod_i|1-tb_i/z|^\beta$ gives an error bounded by
$C|t||z|^{\alpha-3}$ after division by $t^2$.
Its integral between $C_b|t|$ and $1$ tends to zero for every
$\alpha>1$; its weighted integral beyond $1$ does so as well.
Here $C_b$ is chosen larger than $2\max_i|b_i|$.

The density in the correlation integral before inserting the factors
$g_t(y_j-E)$ is bounded by a constant times a product of
polynomial-Gaussian weights, by \eqref{eq:univ-confinement}.
Applying \eqref{eq:univ-L1-taylor} to each factor, and using the finite
product telescoping identity, therefore gives a valid expansion of that
integral.  The coefficient is the sum obtained by replacing one
$|y_j-E|^\alpha$ by
$-\beta B(b)|y_j-E|^{\alpha-2}/2$.
These integrals are finite because $\alpha>1$.
With $t=\ell_N\varepsilon$, the expectation in
\eqref{eq:univ-exact-ratio} consequently equals
\[
 1-\frac{\beta B(b)}2\varepsilon^2
      \mathbb E_{Q_{N,m,E}}S_{2,N}+o_N(\varepsilon^2).
\]
Multiplying the two expansions proves the proposition.
\end{proof}

\subsection{Unconditional coefficient universality in the unitary class}
\label{subsec:unitary-universality}

For $\beta=2$, determinant structure gives a stronger conclusion without
assumptions on the fused environmental law.  We require $U\in C^4(\mathbb R)$,
\eqref{eq:univ-confinement}, and that its equilibrium measure have a
continuous strictly positive density in an open interval containing $E$.
These are hypotheses on the original potential and its ordinary equilibrium
measure; no convergence or inverse-moment assumption on $Q_{N,m,E}$ is
imposed.  The result permits nonanalytic potentials and several support intervals,
provided the equilibrium density has this local bulk regularity.

\begin{theorem}[Unitary fusion coefficients for general potentials]
\label{thm:unitary-universality}
Let $\beta=2$, let $U\in C^4(\mathbb R)$ satisfy
\eqref{eq:univ-confinement}, and suppose its equilibrium measure is
absolutely continuous with a continuous strictly positive density on an
open interval containing $E$. Let $m\ge2$ be fixed and use the exact
normalization $A_N(E)$ in \eqref{eq:univ-amplitude}.  For every compact set
of centered real profiles $b$, the functions $R_N(t;b)$ converge locally in
$C^4$ in $t$ to $R_{2,m}(t;b)$, uniformly in $b$.  Profiles with coincident
coordinates are interpreted by continuous extension after dividing out
$\Delta_m(b)^2$.  Moreover,
\[
 A_N(E)\longrightarrow C_2^{(m)}.
\]
Writing
\[
 R_N(t;b)=1-c_N(b)t^2+d_N(b)t^3+e_N(b)t^4+o_N(t^4),
\]
one has, uniformly on such compact profile sets,
\begin{align*}
 c_N(b)&\longrightarrow
 \kappa_{2,m}(b):=\frac{m p_2(b)}{2(4m^2-1)},\\
 d_N(b)&\longrightarrow0,\\
 e_N(b)&\longrightarrow
 \kappa_{4,m}(b):=
 \frac{(2m^2-3)p_2(b)^2-mp_4(b)}
      {16(4m^2-1)(4m^2-9)}.
\end{align*}
More precisely, for each compact profile set $\mathcal K$ there is a
function $\omega_{\mathcal K}(s)\to0$ as $s\downarrow0$ such that,
for all $N\ge m+1$, $b\in\mathcal K$, and sufficiently small $|t|$,
\[
 \left|R_N(t;b)-1+c_N(b)t^2-d_N(b)t^3-e_N(b)t^4\right|
 \le |t|^4\omega_{\mathcal K}(|t|).
\]
Consequently, for every sequence of nonzero real numbers $t_N\to0$,
\[
 R_N(t_N;b)=1-\kappa_{2,m}(b)t_N^2+o(t_N^2).
\]
\end{theorem}

\begin{proof}
Let $p_{j,N}$ be the orthonormal polynomials for $e^{-NU(x)}\,dx$, and set
\[
 P_N(u,v)=\sum_{j=0}^{N-1}p_{j,N}(u)p_{j,N}(v),\qquad
 k_N=e^{-NU(E)}P_N(E,E).
\]
Theorem~1.1, Remark~(iv), equation~(1.13), and equation~(6.1) of
Levin--Lubinsky~\cite{LL} apply with $Q=U/2$, $h=1$, and $\Sigma=\mathbb R$.
They give $k_N/N\to\varrho_U(E)$ and, locally uniformly for $a,b\in\mathbb C$,
\[
 \frac{P_N(E+a/k_N,E+b/k_N)}{P_N(E,E)}
 \exp\!\left\{-\frac{NU'(E)}{2k_N}(a+b)\right\}
 \longrightarrow \frac{\sin\pi(a-b)}{\pi(a-b)}.
\]
Thus the entire two-variable kernels
\[
 L_N(z,w)=\ell_N e^{-NU(E)}
 \exp\!\left\{-\frac{N\ell_NU'(E)}2(z+w)\right\}
 P_N(E+\ell_Nz,E+\ell_Nw)
\]
converge locally uniformly on $\mathbb C^2$ to
\[
 K(z,w)=\frac{\sin((z-w)/2)}{\pi(z-w)}.
\]
Here we used $\ell_Nk_N\to1/(2\pi)$.  The analytic exponential in $L_N$
is a linear correction of the external field.  It does not require
analytic continuation of $U$.

For $z,w\in\mathbb C^m$, initially with distinct coordinates, put
\[
 H_N(z,w)=\frac{\det[L_N(z_i,w_j)]_{i,j=1}^m}
                    {\Delta_m(z)\Delta_m(w)}.
\]
This quotient extends to an entire function.  More precisely, elementary
Newton interpolation row and column operations give
\[
 H_N(z,w)=\det\!\left[
 L_N[z_1,\ldots,z_i;w_1,\ldots,w_j]
 \right]_{i,j=1}^m,
\]
where the entries are divided differences in the two variables.  On fixed
compact sets, the latter admit the contour formula
\[
 L_N[z_1,\ldots,z_i;w_1,\ldots,w_j]
 =\frac1{(2\pi\mathrm i)^2}\oint\!\oint
 \frac{L_N(\zeta,\eta)\,d\zeta\,d\eta}
      {\prod_{r=1}^i(\zeta-z_r)\prod_{s=1}^j(\eta-w_s)},
\]
with two fixed contours enclosing all the nodes.  This formula proves that
$H_N\to H$ locally uniformly, including on all collision diagonals, where
$H$ is defined in the same way from $K$.  Cauchy's formula also gives
convergence of every fixed derivative of $H_N$.

To recover the actual weighted correlation, define for real $x$
\[
 \delta_N(x)=U(E+\ell_Nx)-U(E)-\ell_NU'(E)x,\qquad
 G_N(x_1,\ldots,x_m)=\exp\!\left\{-N\sum_{i=1}^m \delta_N(x_i)\right\}.
\]
The determinantal correlation formula gives the exact identity
\[
 \frac{\widehat\rho_N^{(m)}(x)}{\Delta_m(x)^2}
 =G_N(x)H_N(x,x).
\]
Since $U\in C^4$ and $\ell_N=O(N^{-1})$, $G_N\to1$ in $C^4$ on every
compact real set.  For example, $N\delta_N$ and its first derivative are
$O(N^{-1})$, while its derivatives of orders $2,3,4$ are respectively
$O(N^{-1}),O(N^{-2}),O(N^{-3})$.  Therefore
$G_N(x)H_N(x,x)\to H(x,x)$ in $C^4$ on compact real sets.

At $x=0$, the left-hand side is its full collision limit, so
$A_N(E)=H_N(0,0)$.  Positivity follows either from
\eqref{eq:univ-amplitude} or from the Gram determinant of the first $m$
polynomial derivatives.  The sine-kernel collision identity gives
$H(0,0)=C_2^{(m)}>0$.  Consequently
\[
 R_N(t;b)=\frac{G_N(tb)H_N(tb,tb)}{H_N(0,0)}
 \longrightarrow \frac{H(tb,tb)}{C_2^{(m)}}=R_{2,m}(t;b)
\]
in the asserted $C^4$ sense.  Symmetry in the coordinates makes the
linear coefficient proportional to $\sum_i b_i=0$; hence it vanishes
for every $N$, without a reflection-symmetry assumption on $U$.
The limiting coefficients now follow from Theorem~\ref{thm:main} at
$\beta=2$.  The $C^4$ convergence makes the fourth derivatives an
equicontinuous family near zero, which gives the uniform Taylor remainder.
Finally the convergent quadratic coefficients and a uniform bound on the
third derivatives imply the stated joint second-order expansion.
\end{proof}

\begin{remark}[Scope of the joint limit]
The last assertion uses the exact finite-$N$ amplitude $A_N(E)$.
Replacing it by $C_2^{(m)}$ in a second-order joint limit requires the
additional rate $A_N(E)/C_2^{(m)}-1=o(t_N^2)$.  Also, convergence of
$c_N,d_N,e_N$ by itself gives no joint fourth-order expansion with all
coefficients replaced by their limits: that would require, for example,
$c_N-\kappa_{2,m}=o(t_N^2)$ and $d_N=o(t_N)$.  These rate questions are
separate from the unconditional coefficient convergence proved above.
\end{remark}

\subsection{A conditional criterion for arbitrary beta}

We return to $\beta>0$, $m\beta>1$, and the $C^2$ assumptions of
Subsection~\ref{subsec:finite-general}. The following criterion requires convergence and bounds
for the singularly tilted environmental law. For arbitrary $\beta$,
these are additional hypotheses. Theorem~\ref{thm:unitary-universality}
proved the coefficient limits in the unitary class directly and did not
use this criterion.

\begin{theorem}[Conditional transfer of the quadratic coefficient]
\label{thm:univ-conditional-transfer}
Let $r_*(x)$ be the one-point intensity of
$\mathsf{HP}_{\beta,m\beta/2}$ in the normalization of
\eqref{eq:univ-unfolding}.  In addition to the assumptions above,
suppose that the following properties hold:
\begin{enumerate}
 \item For every compact $K\subset\mathbb R\setminus\{0\}$,
 \begin{equation}
   \int_K|r_N(x)-r_*(x)|\,dx\longrightarrow0.
   \label{eq:univ-local-convergence}
 \end{equation}
 \item There is a constant $C_0$, independent of $N$, such that
 \begin{equation}
   r_N(x)\leq C_0|x|^{m\beta}
   \quad\text{for almost every }0<|x|\leq1.
   \label{eq:univ-origin-bound}
 \end{equation}
 \item The inverse-square tails are uniformly negligible:
 \begin{equation}
   \lim_{R\to\infty}\sup_{N\geq m+1}
      \int_{|x|>R}\frac{r_N(x)}{x^2}\,dx=0.
   \label{eq:univ-tail-bound}
 \end{equation}
\end{enumerate}
Then
\begin{equation}
 \lim_{N\to\infty}c_N(b)
 =\frac{\beta^2\mathcal V(b)}
        {8(m\beta-1)(2m+1)}.
 \label{eq:univ-coefficient-limit}
\end{equation}
\end{theorem}

\begin{proof}
By the intensity formula for nonnegative linear statistics,
\[
 \mathbb E_{Q_{N,m,E}}S_{2,N}
 =\int_{\mathbb R}\frac{r_N(x)}{x^2}\,dx.
\]
For $0<\eta<1$, the contribution of $|x|<\eta$ is bounded uniformly in
$N$ by
\begin{equation}
 \int_{|x|<\eta}\frac{r_N(x)}{x^2}\,dx
 \leq\frac{2C_0}{m\beta-1}\eta^{m\beta-1}.
 \label{eq:univ-small-cutoff}
\end{equation}
On $\eta\leq|x|\leq R$, convergence of the inverse-square integrals
follows from \eqref{eq:univ-local-convergence}.
The limiting intensity inherits the same local bound on annuli, hence
also satisfies \eqref{eq:univ-small-cutoff}.
The tails for $r_N$ are controlled by \eqref{eq:univ-tail-bound}.
The inverse-square moment identity of \cite[Theorem 1.3]{FangSecond}, valid throughout $m\beta>1$, is
\begin{equation}
 \int_{\mathbb R}\frac{r_*(x)}{x^2}\,dx
 =\frac{m\beta}{4(m\beta-1)(2m+1)};
 \label{eq:univ-HP-moment}
\end{equation}
in particular its tails vanish as well.
Taking first $N\to\infty$, then $\eta\downarrow0$ and $R\to\infty$,
proves convergence of the finite inverse-square moments to
\eqref{eq:univ-HP-moment}.
Finally $N\ell_N^2=O(N^{-1})$, so the external-potential contribution
in \eqref{eq:univ-finite-coefficient} vanishes.  Substituting
\eqref{eq:univ-HP-moment} there and using
$\mathcal V(b)=mB(b)$ proves \eqref{eq:univ-coefficient-limit}.
\end{proof}

Theorem~\ref{thm:univ-conditional-transfer} concerns the limit of finite-$N$
coefficients, equivalently the iterated limit with fusion taken first.
Its hypotheses alone supply neither a rate
in \eqref{eq:univ-coefficient-limit} nor a remainder uniform in $N$ in
\eqref{eq:univ-finite-expansion}.  Using this conditional criterion for a simultaneous limit
$\varepsilon=\varepsilon_N\to0$ requires additional uniform control of
that remainder. The unitary argument above supplies such control under
its stated assumptions.  For example, it would suffice to establish
\[
 \left|\frac{1-R_N(\varepsilon;b)}{\varepsilon^2}-c_N(b)\right|
 \leq\omega(|\varepsilon|),\qquad \omega(t)\longrightarrow0,
\]
uniformly for all sufficiently large $N$.

The use of the exact amplitude $A_N(E)$ is essential to the formulation.
Replacing it by the universal leading amplitude $C_\beta^{(m)}$ in a
simultaneous limit additionally requires
$A_N(E)/C_\beta^{(m)}-1=o(\varepsilon_N^2)$.
Centering the profile eliminates its first variation exactly;
an uncentered profile also samples the variation of the collision
amplitude as the common center moves.
For arbitrary $\beta$, proving local convergence
\eqref{eq:univ-local-convergence} together with both uniform bounds
\eqref{eq:univ-origin-bound}--\eqref{eq:univ-tail-bound} for general
potentials remains a separate universality problem.

\section{Further questions}
Higher even coefficients involve additional mixed inverse moments and a larger Ward system. The one-particle integrability scale suggests $m\beta>2k-1$ as a natural range for the $2k$th inverse-moment calculation. A complete higher-order theorem also requires an expectation--Taylor argument with compatible moment exponents and control of the resulting mixed terms.

At $m\beta=3$, identifying the coefficient of a possible $\varepsilon^4\log(1/|\varepsilon|)$ term requires a collision-scale argument. The interpolation method of \cite{FangCritical} supplies a possible starting point, but the pole of \eqref{eq:main-coefficient} alone is not a proof. Exceptional profiles and parameter values must be treated separately when the residue vanishes.

For the unitary class, Theorem~\ref{thm:unitary-universality} establishes fusion-coefficient universality for general $C^4$ confining potentials at regular bulk points, through fourth order. For arbitrary $\beta$, the conditional criterion isolates three tasks: proving local $L^1$ convergence of the fused one-point density away from zero, establishing a uniform near-origin bound, and controlling the inverse-square tails uniformly. All three are required to obtain an unconditional theorem by that route. A second-order joint limit also needs a remainder uniform in the particle number; the unitary theorem provides one with the exact finite-particle amplitude. A joint fourth-order expansion with universal lower-order coefficients requires additional convergence rates, as explained in Subsection~\ref{subsec:unitary-universality}.

\section*{AI-assisted preparation}
OpenAI ChatGPT assisted with exploratory calculations, symbolic checks, proof development, and drafting. This version is prepared for the author's review.

\end{document}